\documentclass[]{article} 

\usepackage{mathtools}
\usepackage{amsmath}
\usepackage{cases}
\usepackage{float}
\usepackage[toc,page]{appendix}

\usepackage{kantlipsum}
\usepackage{tikz}
\usepackage{pgf,tikz}
\usepackage{amsthm}

\usetikzlibrary{arrows}
\usetikzlibrary[patterns]

\usepackage{amsmath, amssymb, amsthm}
\usepackage{color}

\theoremstyle{plain}
\newtheorem{thm}{Theorem}[section]

\newtheorem*{prop*}{Proposition}
\newtheorem{lem}[thm]{Lemma}
\newtheorem*{lem*}{Lemma}
\newtheorem{cor}[thm]{Corollary}
\newtheorem{claim}[thm]{Claim}

\theoremstyle{definition}

\newtheorem{dfn}[thm]{Definition}
\newtheorem{conj}[thm]{Conjecture}

\theoremstyle{remark}

\newtheorem{innercustomgeneric}{\customgenericname}
\providecommand{\customgenericname}{}
\newcommand{\newcustomtheorem}[2]{%
  \newenvironment{#1}[1]
  {%
   \renewcommand\customgenericname{#2}%
   \renewcommand\theinnercustomgeneric{##1}%
   \innercustomgeneric
  }
  {\endinnercustomgeneric}
}

\newcustomtheorem{customthm}{Theorem}
\newcustomtheorem{customlemma}{Lemma}
\newcustomtheorem{customclaim}{Claim}

\newcommand{\beq}{\begin{equation}}
\newcommand{\eeq}{\end{equation}}

\providecommand{\keywords}[1]{\textbf{\textit{Keywords: }} #1}

\usepackage{listings,xcolor}

\usepackage[utf8]{inputenc}

\DeclareFixedFont{\ttb}{T1}{txtt}{bx}{n}{12} 
\DeclareFixedFont{\ttm}{T1}{txtt}{m}{n}{12}  

\usepackage{color}
\definecolor{deepblue}{rgb}{0,0,0.5}
\definecolor{deepred}{rgb}{0.6,0,0}
\definecolor{deepgreen}{rgb}{0,0.5,0}

\newcommand\pythonstyle{\lstset{
language=Python,
basicstyle=\ttm,
otherkeywords={self},             
keywordstyle=\ttb\color{deepblue},
emph={MyClass,__init__},          
emphstyle=\ttb\color{deepred},    
stringstyle=\color{deepgreen},
frame=tb,                         
showstringspaces=false            %
}}

\lstnewenvironment{python}[1][]
{
\pythonstyle
\lstset{#1}
}
{}

\newcommand\pythoninline[1]{{\pythonstyle\lstinline!#1!}}

\date{}
\title{\vspace{-0.7cm} {On the number of cliques in graphs with a forbidden subdivision or immersion}}

\author{
Jacob Fox\thanks{Department of Mathematics, Stanford University, Stanford, CA 94305. Email: {\tt jacobfox@stanford.edu}. Research supported by a Packard Fellowship, by NSF Career Award DMS-1352121 and by an Alfred P. Sloan Fellowship.}
\and
Fan Wei\thanks{Department of Mathematics, Stanford University, Stanford, CA 94305. Email: {\tt fanwei@stanford.edu}.}
}

\begin{document}

\maketitle

\begin{abstract}
How many cliques can a graph on $n$ vertices have with a forbidden substructure? Extremal problems of this sort have been studied for a long time. This paper studies the maximum possible number of cliques in a graph on $n$ vertices with a forbidden clique subdivision or immersion. We prove for $t$ sufficiently large that every graph on $n \geq t$ vertices with no $K_t$-immersion has at most $2^{t+\log^2 t}n$ cliques, which is sharp apart from the $2^{O(\log^2 t)}$ factor. We also prove that the maximum number of cliques in an $n$-vertex graph with no $K_t$-subdivision is at most $2^{1.817t}n$. This improves on the best known exponential constant by Lee and Oum. We conjecture that the optimal bound is $3^{2t/3 +o(t)}n$, as we proved for minors in place of subdivision in earlier work. 
\end{abstract}

\section{Introduction} \label{sec:intro}

Many questions in extremal graph theory ask to estimate the maximum number of cliques (maybe of a given order) a graph on a given number of vertices can have with a forbidden substructure. For example, Tur\'an's theorem \cite{Tu}, a cornerstone result of extremal graph theory determines the maximum number of edges in a $K_t$-free graph on $n$ vertices, and the only graph to achieve this bound is the balanced complete $(t-1)$-partite graph. This result was later extended by Zykov \cite{Zy}, who determined that the same graph maximizes the number of cliques amongst $K_t$-free graphs on $n$ vertices. 

Questions of this sort for forbidden minors have also been studied for a long time. For example, Mader \cite{Ma,Ma1} showed that for each positive integer $t$ there is a constant $C(t)$ such that every $K_t$-minor-free graph on $n$ vertices has at most $C(t)n$ edges.  Kostochka \cite{Ko,Ko1} and Thomason \cite{Th} independently proved that $C(t)=\Theta(t\sqrt{\log t})$. Thomason \cite{Th} later proved that $C(t)=(\alpha+o(1))t\sqrt{\log t}$, where $\alpha=0.319...$ is an explicit constant.
Motivated by various algorithmic applications (see, e.g., \cite{DKT}, and \cite{RW}), the problem of maximizing the number of cliques in a $K_t$-minor-free graph on $n$ vertices has been studied by various researchers. Norine, Seymour, Thomas, and Wollan \cite{NSTW} and independently Reed and Wood \cite{RW} showed that for each $t$ there is a constant $c(t)$ such that every graph on $n$ vertices with no $K_t$-minor has at most $c(t)n$ cliques. Lee and Oum \cite{LO} proved a conjecture of Wood \cite{Wo} by showing that $c(t) < c^t$ for some absolute constant $c$. The authors in \cite{minor} proved that $c(t)=3^{2t/3 + o(t)}$, which is tight up to the lower order factor. They further generalized the result to graphs with a forbidden connected graph $H$ as a minor, or within a minor-closed family of graphs which is closed under disjoint union. 

Analogous questions for graphs with a forbidden subdivision (also known as topological minors) or immersion have also received much attention. A {\it weak $H$-immersion} in $G$ is a one-to-one mapping $\phi:V(H) \to V(G)$ together with a collection of edge-disjoint paths in $G$, one for each edge $uv$ of $H$ which has end vertices $\phi(u)$ and $\phi(v)$. If none of the internal vertices of the paths are in the image $\phi(V(H))$, then this is a {\it strong $H$-immersion}. If the paths in a strong $H$-immersion are internally vertex disjoint, then the mapping and the paths form an {\it $H$-subdivision} in $G$.

Note that, within a factor two, determining the minimum number of edges (or equivalently, the average degree) needed to guarantee a graph on a given number of vertices has a $K_t$-immersion is the same as determining the minimum degree condition needed to guarantee a $K_t$-immersion. Lescure and Meyniel in \cite{LM} observed that the minimum degree needed to guarantee a strong $K_t$-immersion has to be at least $t-1$, and DeVos et al.~\cite{DeVos7} proved this bound is sharp for $t \leq 7$. An example of Seymour (see \cite{DeVos7}) shows that there are graphs with minimum degree $t-1$ but without a $K_t$-immersion for $t \geq 10$. DeVos et al. \cite{Devos} proved that every graph with minimum degree $200t$ contains a strong $K_t$-immersion. The minimum degree condition to guarantee a strong $K_t$-immersion was recently improved by Dvo\v r\'ak and Yepremyan \cite{DY} to $11t+7$.

The following theorem determines, up to a lower order factor, the maximum number of cliques in a graph on $n$ vertices without a $K_t$-immersion. 

\begin{thm} \label{immerse}
Every graph on $n$ vertices with no strong $K_t$-immersion has at most $2^{t + {\log_2^2} t}n$ cliques.  This bound is sharp up to a factor $2^{O(\log^2 t)}$.
\end{thm}
Theorem \ref{immerse} also holds for weak immersion instead of strong immersion. As a strong immersion is also a weak immersion, the upper bound on the number of cliques in Theorem \ref{immerse} also applies to weak $K_t$-immersion. The following construction shows that for $n \geq t-2$, there is a graph on $n$ vertices with no weak $K_t$-immersion and has $2^{t-2}(n-t+3)$ cliques. Begin with a clique $K$ on $t-2$ vertices. For the remaining $n-t+2$ vertices, its neighborhood is $K$.  Since $|K|=t-2<t$, the vertices not in $K$ have degree less than $t-1$, and the $t$ end vertices of a $K_t$-immersion must have degree at least $t-1$, then this graph does not have a $K_t$-immersion. The cliques in this graph are formed by taking any subset of $K$ (and there are $2^{t-2}$ such subsets) and at most one of the remaining $n-t+2$ vertices, for which there are $n-t+3$ possibilities, giving a total of $2^{t-2}(n-t+3)$ cliques. We conjecture that this is the maximum number of cliques a graph on $n \geq t-2$ vertices can have if it contains no $K_t$-immersion.

The analogous problems for forbidden subdivisions appears to be more challenging. Koml\'os and Szemer\'edi \cite{KS} and Bollob\'as and Thomason \cite{BT} solved a longstanding conjecture of Erd\H{o}s, Hajnal, and Mader by showing  that every graph on $n$ vertices with no $K_t$-subdivision has $O(t^2n)$ edges, which is tight apart from the implied constant factor. Our next theorem significantly improves the exponential constant obtained by Lee and Oum \cite{LO} for the number of cliques in a graph with a forbidden clique subdivision. 

\begin{thm} \label{main2} 
Every graph on $n$ vertices with no $K_t$-subdivision has at most $2^{1.817t + o(t)}n$ cliques. 
\end{thm} 

The proof of Theorem \ref{main2} has some similarities to the proof of Theorem \ref{immerse}, but is much more involved. We first establish a weaker result, giving an upper bound on the number of cliques of $2^{3t+o(t)}n$ in $n$-vertex graphs with no $K_t$-subdivision. We then discuss how optimization of the computation gives the bound claimed in Theorem \ref{main2}. The same lower bound construction for minors  discovered by Wood also forbids clique subdivisions of the same size and shows that the constant factor 1.817 cannot be improved to less than $\frac{2}{3}\log_2 3\approx 1.06$. We conjecture that the constant $\frac{2}{3}\log_2 3$ is best possible as it is for minors. This conjecture as well as some possible approaches to improve the bound further are discussed in the concluding remarks. 

The proofs of Theorems \ref{immerse} and \ref{main2} utilize a method to enumerate all the cliques, which we call the \emph{peeling process} and is explained in detail in the next section. This method strongly resembles the powerful container method. Early ideas along these lines appear in the works of Kleitman and Winston \cite{KW} and Sapozhenko \cite{Sa1,Sa2,Sa3}. However, it was only fairly recently that the power and generality of the idea was fully realized and developed by Balogh, Morris, and Samotij \cite{BMS} and independently by Saxton and Thomason \cite{ST}.  In the last few years, many applications of this technique were discovered in probabilistic and extremal combinatorics, see the recent survey \cite{BMS2} for more.


\vspace{0.2cm}

\noindent {\bf Organization:} In the next section we prove Theorem \ref{immerse} on the number of cliques in graphs with no strong $K_t$-immersion. We prove Theorem \ref{main2} on the number of cliques in graphs with no $K_t$-subdivision in Section \ref{subsec}. We finish with some concluding remarks. All logarithms are base $2$ unless otherwise noted. For the sake of clarity of presentation, we sometimes omit floor and ceiling signs.

\section{Counting cliques in graphs with no $K_t$-immersion}\label{immersesec}
In the first subsection we prove a lemma which gives a bound on the minimum degree of a graph on $n$ vertices with no strong $K_t$-immersion. In the following subsection, we discuss a simple technique to enumerate the cliques of a graph. We then also use the lemma, and the result of \cite{Devos,DY} which bounds the minimum degree in a graph with no strong $K_t$-immersion. All these ingredients together, with appropriate analysis, will give Theorem \ref{immerse}.

\subsection{Clique immersions in dense graphs}
In this subsection, we give a sufficient condition on the minimum degree for a graph on $n$ vertices to contain a strong $K_t$-immersion. We use the term {\it missing edge} for a nonadjacent pair of distinct vertices and the term {\it missing degree} for the number of missing edges incident to a particular vertex. 

\begin{lem}\label{immersionnumber}
If a graph $G=(V,E)$ with $n$ vertices has a vertex subset $T$ with $|T|=t$ and every vertex in $T$ has missing degree less than $(n-t+2)/2$, then $G$ has a strong $K_t$-immersion with $T$ being the set of end vertices. In particular, if a graph with $n$ vertices and maximum missing degree less than $(n-t+2)/2$, then it contains a strong $K_t$-immersion. 
\end{lem}
\begin{proof}

Let $|S|=V \setminus T$, so $|S| = n-t$. To create a strong $K_t$-immersion with $T$ the set of end vertices, we need to connect the nonadjacent pairs of distinct vertices in $T$ by edge-disjoint paths with internal vertices in $S$. We prove more, that it can be done with only paths of length two. Each of these paths have a single internal vertex, which is in $S$.

The proof is by induction on the number  $m$ of nonadjacent pairs of vertices in $T$. If there are no nonadjacent pairs, then $T$ is a clique and forms a strong $K_t$-immersion, thus settling the base case $m=0$. Suppose we have established the result for $m-1$, and $T$ has $m$ nonadjacent pairs of distinct vertices. Consider a pair of nonadjacent distinct vertices $u,v \in T$. The number of vertices other than $u$ and $v$ which are not adjacent to $u$ or $v$ is less than $2(n-t+2)/2 - 2 =  n - t = |S|$. Thus, there is a vertex $w \in S$ which is adjacent to both $u$ and $v$. Let $G'=(V,E')$ be the graph formed from $G$ by adding the edge $(u,v)$ and deleting the edges $(u,w)$, $(w,v)$. Note that the degrees of the vertices in $T$ is the same in $G'$ as in $G$ and the number of missing edges in $T$ in $G'$ is $m-1$. Thus, by the induction hypothesis, there is a strong $K_t$-immersion in $G'$ with $T$ the set of end vertices and for paths using all edges in $T$ and paths of length two to connect the nonadjacent vertices in $T$. Replacing the path in this strong $K_t$-immersion which is the single edge $(u,v)$ by the length two path with edges $(u,w)$ and $(w,v)$, we obtain a strong $K_t$-immersion in $G$, completing the proof by induction. 
\end{proof}

We remark that the maximum missing degree condition in the lemma above is tight. In other words, there is a graph with maximum missing degree $(n-t+2)/2$ which has a vertex subset $T$ with $|T| = t$ such that there is no strong $K_t$-immersion with $T$ being the set of end vertices. The construction is as follows. Let $T$ be a set of $t$ vertices with only one missing edge: $u, v \in T$ are not adjacent. Let $S_1, S_2$ be two cliques with $(n-t)/2$ vertices each. There are no edges between $S_1, S_2$. All vertices in $S_1$ are connected to all vertices in $T$ except $v$; all vertices in $S_2$ are connected to all vertices in $T$ except $u$. Let $G$ be the resulting graph with vertex set $T \cup S_1 \cup S_2$. It is easy to check that the maximum missing degree in $G$ is $(n-t)/2+1$. To show that there is no strong $K_t$-immersion in $G$ with $T$ the set of end vertices, notice that $u,v \in T$ are not adjacent, and their neighborhoods in $V(G) \setminus T$ are disjoint. So there is no path using interior vertices in $V(G) \setminus T$ to connect $u,v$. Thus there is no strong $K_t$-immersion with $T$ being the set of end vertices. However, in this graph there is a weak $K_t$-immersion with $T$ being the set of end vertices.

\subsection{Proof of Theorem \ref{immerse}}\label{finimmersion}

Lee and Oum \cite{LO} (and in earlier works such as \cite{KW}) present a simple greedy algorithm to enumerate all the cliques in a graph $G$, which we call the peeling process. It is useful in bounding the number of cliques in a graph. We used this method in \cite{minor} to bound the number of cliques in graphs with a forbidden minor. We also utilize this method here, and thus copy the description of the process as it is written in \cite{minor} here. 

\vspace{0.3cm}

\noindent {\bf Peeling Process}  

Given a graph $G$, pick a minimum degree vertex $v_1$. We first enumerate all cliques containing $v_1$. We accomplish this by applying the algorithm to the induced subgraph of $G$ with vertex set being the neighborhood $N(v_1)$ to obtain all cliques within $N(v_1)$, and add $v_1$ to each of these cliques. Once we have exhausted all cliques containing $v_1$, we delete it from the graph and continue the algorithm to the induced subgraph of the remaining vertices. 


In this way, every clique $K$ in $G$ has an ordering $v_1,\ldots,v_s$ of its vertices so that the following holds. Let $G_0=G$. After deleting vertices one at a time of degree smaller than the degree of $v_1$, we obtain an induced subgraph $G_1$ that contains $K$ in which $v_1$ has the minimum degree.  After picking $v_1,\ldots,v_i$, we delete from $G_i$ vertex $v_i$ and its nonneighbors, and vertices one at a time of degree smaller than that of $v_{i+1}$ and obtain an induced subgraph $G_{i+1}$ of $G_i$ that contains $K \setminus \{v_1,\ldots,v_i\}$ and in which $v_{i+1}$ has the minimum degree. 

The peeling process allows us to focus on a remaining induced subgraph which is small or quite dense and is easier to analyze. There are not many vertices of any given clique not in this remaining induced subgraph. These few vertices contribute a relatively small factor to the total number of cliques.

\begin{proof}[Proof of Theorem \ref{immerse}]
Similar to our approach in in \cite{minor}, we first apply the peeling process described above to the graph $G$. Let $K$ be a clique on $s$ vertices. Let $n_i$ denote the number of vertices in $G_i$.  Let $r=r(K)$ be the least positive integer such that $n_r \leq t + 1$ or $r=s$.

We first give a bound on $r$. The results of \cite{Devos,DY} implies, as the simple graph $G$ does not contain a $K_t$-immersion, that every subgraph has a vertex of degree at most $d=11t+8 < 20t$.  Hence $n_2 \leq d$.

As $G$ and hence $G_i$ does not contain a strong $K_t$-immersion, Lemma \ref{immersionnumber} implies that the maximum missing degree in $G_i$ has to be at least $(n_i-t+2)/2$. 
Thus by peeling off the vertex with the maximum missing degree and its non-neighbors, $n_{i+1} \leq n_i - 1 - (n_i-t+2)/2=(n_i+t)/2 - 2$. In particular, $n_{i+1} - t < (n_i - t)/2$.
Since $n_1 \leq d < 20t$ and $n_{r-1} > t + 1$ unless $r=s < t$, we have that $r \leq 1+\log_2 \frac{n_1 - t }{n_{r-1} - t}  < 1+\log_2 (19t) < \log_2 t + 6:= r_0$. 

We next give a simple bound on the number of choices for $v_1,\ldots,v_r$. We have $n_0=n$ choices for $v_1$. Since $v_2, \dots, v_r$ has to be chosen as vertices in $G_1$, the number of choices for $v_2,\ldots,v_r$ after having picked $v_1$ is at most 
$${|G_2|-1 \choose \leq r_0} \leq r_0 \binom{20t}{r_0} \leq r_0 \left(\frac{e 20t}{r_0}\right)^{r_0} <   2^{\log^2 t - 2},$$
where the last inequality follows from substituting in the value of $r_0$ and using that $t$ is sufficiently large. 
Hence, there are at most $n2^{\log^2 t -2}$ choices for $v_1,\ldots,v_r$. 

Recall that $G$ has $n$ vertices and no $K_t$-immersion and our goal is to bound the number of cliques in $G$. We have already bounded the number of choices for the first $r$ vertices, and it suffices to bound the number of choices for the remaining vertices. We split the cliques into two types: those with $r<s$ and $n_r\leq t+1$, and those with  $r=s$.

We first bound the number of cliques with $r<s$ and $n_r \leq t+1$. As there are at most $t+1$ possible remaining vertices to include in the clique after picking $v_1,\ldots,v_r$, there are at most $2^{t+1}$ ways to extend these vertices. Thus there are at most $n2^{\log^2 t -2}2^{t+1}=n2^{t+\log^2 t -1}$ cliques of the first type. 

We next bound the number of cliques with $r=s$. We saw that this is at most $n2^{\log^2 t -2}$. 
Adding up all possible cliques, we get at most $n2^{t+\log^2 t -1}+n2^{\log^2 t -2} < n2^{t+\log^2 t}$ cliques in $G$, completing the proof. 

\end{proof}


\section{Counting cliques in graphs with no $K_t$-subdivision}\label{subsec}
In this section we prove Theorem \ref{main2}. It has some similarities to the proof of Theorem \ref{immerse}, but is substantially more involved and uses several additional ideas. We apply the peeling process to reduce the problem of bounding the number of cliques in any graph on $n$ vertices without a $K_t$-subdivision to bounding the number of cliques in a graph without a  $K_t$-subdivision that is \emph{very dense}, meaning that the minimum degree is $n-o(n)$. More precisely, the main factor in $t$ in the bound comes from the cliques in a very dense induced subgraph, and a lower order factor comes from the other vertices. 

We determine some relevant graph parameters to bound the size of the largest clique subdivision in a very dense graph. We then find a recursive bound on the number of cliques involving these graph parameters, and analyze this recursive bound to get the desired bound for very dense graphs. 




The proof is organized in the following way. In Subsection \ref{count2}, we show that the order of the largest clique subdivision in a very dense  graph on $n$ vertices can be approximated in terms of other graph parameters.  Then, in Subsection \ref{rec2}, we use these parameters to give a recursive bound for the number of cliques in a very dense graph. Finally, in Subsection \ref{fin2}, we combine the peeling process described in Section \ref{immersesec} and the recursive  bound we found in Subsection \ref{rec2} to prove a bound in the number of cliques in any graph without a $K_t$-subdivision.

\subsection{Clique subdivision number of very dense graphs}\label{count2}
For a graph $G$, the clique subdivision number $\sigma(G)$ is the maximum $h$ for which $G$ contains a $K_h$-subdivision. We will use some other graph parameters to bound $\sigma(G)$. 

\begin{lem}\label{simplem}
If a graph $G$ with $n$ vertices and minimum degree $n-\Delta-1$ has a vertex subset $T$ of order $t$ and at most $n-t-2\Delta$ missing edges, then $G$ contains a $K_t$-subdivision with the vertices of $T$ as end vertices. 
\end{lem}
\begin{proof}
For each missing edge $(v_1,v_2)$ with both vertices in $T$, we have $|N(v_1)| \geq n- \Delta-1$ and $|N(v_2)| \geq n- \Delta-1$. As $v_1,v_2$ are not adjacent, we therefore have $|N(v_1) \cap N(v_2)| \geq n - 2\Delta$ and hence $$|N(v_1) \cap N(v_2)\cap (V \setminus T)| \geq n - 2\Delta -t.$$ For each missing edge in $T$, we can greedily pick a common neighbor of its vertices outside $T$ to be the internal vertex of a path of length two connecting these two end vertices. We therefore obtain a clique subdivision with $T$ being the set of end vertices, and using edges and paths of length two to connect the end vertices.  
\end{proof}

In a graph $G$ with $n$ vertices, define the graph parameter $t(G)$ to be the maximum $t$ such that $G$ has a vertex subset $T$ of order $t$ with at most $n-t$ edges missing within $T$. The following lemma shows that $t(G)$ and the the clique subdivision number $\sigma(G)$ are close to each other in very dense graphs. 

\begin{lem}\label{subdp}
Every graph $G$ on $n$ vertices with minimum degree $n-\Delta-1$ satisfies 
$$t(G) -\Delta \leq \sigma(G) \leq t(G).$$ 
Furthermore, we can give an upper bound for the maximum missing degree $\Delta$ in terms of an upper bound for the clique subdivision number. If $\sigma(G)<t \leq n$, then $$\Delta \leq \frac{2(n-t)(n-1)}{4(n-1)+t(t-1)}.$$
\end{lem}
\begin{proof}

We first prove the upper bound on $\sigma(G)$. Suppose that $G$ has a $K_{\sigma}$-subdivision, and let $S$ denote the subset of the $\sigma$ end vertices  of this subdivision. For each missing edge within $S$, there is at least one internal  vertex of $V(G) \setminus S$ on the path connecting its end vertices in the subdivision. Since these internal vertices are distinct, and there are $n- \sigma$ vertices not in $S$, the number of edges missing within $S$ is at most $n-\sigma$. It follows that $\sigma(G) \leq t(G)$. 

We next prove the lower bound on $\sigma(G)$. Let $T$ be a vertex subset of $G$ of order $t=t(G)$ with at most $n-t$ edges missing within $T$. Remove  vertices within $T$ one at a time that is an end vertex of at least one missing edge within $T$ until the remaining vertex subset $T'$ is a clique or has cardinality $t-\Delta$. If $T'$ is a clique, then $\sigma(G) \geq |T'| \geq t-\Delta$, and we are done. Otherwise, $T'$ has cardinality $t-\Delta$ and at most $n-t-\Delta=n-(t-\Delta)-2\Delta$ missing edges.  In this case, we can apply Lemma \ref{simplem} and obtain a clique subdivsion with $T'$ the set of end vertices, implying that $\sigma(G) \geq t(G)-\Delta$. 

 We next prove the upper bound on $\Delta$. Since $G$ has maximum missing degree $\Delta$, it misses at most $\frac{\Delta n}{2}$ edges. By averaging over all subsets with $t$ vertices, (which we call a $t$-set for short), there must exists a $t$-set $U$ such that the number of missing edges within $U$ is at most $\frac{\Delta n}{2} \cdot \frac{\binom{t}{2}}{\binom{n}{2}}$. From Lemma \ref{simplem}, if the number of missing edges is at most $n-t-2\Delta$, then $G$ contains a $K_t$-subdivision. Hence, if $G$ does not contain a $K_t$-subdivision, i.e., $\sigma(G)<t$, then $\frac{\Delta n}{2} \cdot \frac{\binom{t}{2}}{\binom{n}{2}} < n-t-2\Delta$. By rearranging, 
 $ \Delta \leq \frac{n-t}{2+\frac{t(t-1)}{2(n-1)}}.$ After simplifying the expression, we obtain $\Delta \leq \frac{2(n-t)(n-1)}{4(n-1)+t(t-1)}$ as in the lemma statement. 
\end{proof}

\subsection{On the number of cliques in very dense graphs}\label{rec2}
Similar to the approach in Section \ref{immersesec}, we will use the peeling process  to reduce the problem of bounding the number of cliques in a graph with no $K_t$-subdivision to the case of dense graphs. Subsection \ref{count2} relates the clique subdivision number of a dense graph to other graph parameters $t(G)$ and $\Delta(G)$ which will be easier to work with. We use these parameters to bound the number of cliques in a dense graph. Specifically, we obtain that the number of cliques in a dense graph $G$ on $n$ vertices with no $K_t$-subdivision is at most $2^{3t+o(t)}n$. A more careful optimization improves the constant factor $3$ in the exponent to $1.817$.

The reason we focus on dense graphs is analogous to that in the clique minor case in \cite{minor}. Because this argument is important, we discuss it here again. By the peeling process described in Section \ref{immersesec}, we will end up with a dense induced  subgraph $G_0$ after a small number of steps. Here the criteria for being dense is that the maximum missing degree $\Delta$ in $G_0$ is $o(t)$. As the number of steps is small, we will show that bounding the number of cliques in the  graph $G$ can be reduced to studying the number of cliques in the dense graph $G_0$ we eventually obtain. We thus reduce the problem to dense graphs, and we study this case here.

The following definition is natural in light of Lemma \ref{subdp}, which helps us bound the number of cliques in a very dense $K_t$-subdivision free graph. 
\begin{dfn}
Let $f(n, x, t)$ be the maximum number of cliques among all graphs with $n$ vertices and every $t$-set misses at least $x$ edges.
\end{dfn}

We now explain the connection between $f(n,x,t)$ and the maximum number of cliques in a very dense graph on $n$ vertices without a $K_t$-subdivision. Since $G_0$ has no $K_t$-subdivision, $\sigma(G_0) \leq t-1$. Lemma \ref{subdp} implies that if $\sigma(G_0) \leq t-1$, then $t(G_0) \leq t -1 + \Delta$. Recall that we are working with dense graph, which means $\Delta = o(t)$, i.e., $t(G_0)$ is asymptotically at most $t$. Thus $t(G_0) \leq (1+o(1)) t$. Every $t(G_0)$-set in $G_0$ misses at least $n-t(G_0)$ edges. We would thus like to prove that, if $|V(G_0)| = n_0$, then $f(n_0, n_0-t, t) \leq 2^{c t + o(t)}$. 
 
The following lemma is the main result in this subsection. It gives a bound on the number of cliques in a very dense graph without a $K_t$-subdivision.
\begin{lem}\label{boundrec}
For $t \leq n \leq 2^{o(t)}$, there is a constant $c$ such that 
\[f(n,n-t,t) \leq 2^{ct+o(t)}.\] 
Furthermore, we can choose $c=3$.
\end{lem}

 We prove this first with $c=3$ and then later discuss how to improve this to $c=1.817$. Equivalently, substituting in $n_0 = Ct$ (where $C$ might depend on $t$), we show 
 \[f(Ct, (C-1)t, t) \leq 2^{ct + o(t)}.\]


It will be convenient to bound a variant of $f(n,x,t)$ which we define next. 
\begin{dfn}
Let $g(m,x,t,d)$ be the maximum number of cliques over all graphs $G$ with $n$ vertices and maximum missing degree $\Delta$ satisfying $n+\Delta \leq m$, $\Delta \leq d$, and every $t$-set, if exists (i.e., in the case $t \leq n$), misses at least $x$ edges. 
\end{dfn}
The next lemma follows trivially from the definition of $g(m,x,t, d)$.

\begin{lem}\label{monotone}
The function $g(m,x,t, d)$ is monotone increasing in $m$, monotone increasing in $t$, monotone decreasing in $x$, and monotone \footnote{Monotone here does not exclude the possibility that the value stays the same.} increasing in $d$.
\end{lem}

The following lemma establishes a useful recursive bound. 

\begin{lem}\label{glemma}
For each $m,x,t,d$, there is $\Delta_1 \leq d$ such that $$g(m,x,t,d) \leq \sum_{i=0}^{\Delta_1} g(m-\Delta_1-i,x,t,\Delta_1) \leq (\Delta_1+1)g(m-\Delta_1,x,t,\Delta_1).$$
\end{lem}
\begin{proof}
Let $G_0$ be a graph realizing $g(m,x,t,d)$ cliques on $n$ vertices and maximum missing degree $\Delta_1$ with $n+\Delta_1 \leq m$ and $\Delta_1 \leq d$, and every $t$-set misses at least $x$ edges. Let $v_1$ be a vertex of maximum missing degree $\Delta_1$ in $G_0$. Having defined $v_1,\ldots,v_{i-1}$, let  $G_{i-1}$ be the induced subgraph of $G$ formed by deleting $v_1,\ldots,v_{i-1}$. Let $v_i$ be a vertex of $G_{i-1}$ of maximum missing degree, which we denote by $\Delta_{i}$. By the choice of the vertices, we have $d \geq \Delta_1 \geq \Delta_2 \geq \ldots$. 

The cliques containing $v_i$ but no $v_j$ with $j<i$ are, together with $v_i$, the cliques contained in the induced subgraph of $G$ without $v_1,\ldots,v_{i}$ and without the $\Delta_i$ non-neighbors of $v_i$ in $G_{i-1}$. The number of  vertices in this induced subgraph of $G_{i-1}$ is $n-i-\Delta_i$ and the maximum missing degree is at most $\Delta_i$. So the sum of the number of vertices and the maximum missing degree in this induced subgraph is at most $(n-i-\Delta_i)+\Delta_i = n-i \leq m-\Delta_1-i$. Further, it inherits the property that every $t$-set misses at least $x$ edges. Hence, the number of these cliques is at most $g(m-\Delta_1-i,x,t,\Delta_i)$.

Also, the number of cliques not containing $v_1,\ldots,v_{\Delta_1}$ is at most $g(m-\Delta_1,x,t,\Delta_1)$. Therefore, the total number of cliques in $G$ is at most \begin{eqnarray*}g(m,x,t,d) & \leq & g(m-\Delta_1,x,t,\Delta_1)+\sum_{i=1}^{\Delta_1} g(m-\Delta_1-i,x,t,\Delta_i) \\ & \leq & 
\sum_{i=0}^{\Delta_1} g(m-\Delta_1-i,x,t,\Delta_1) \\ & \leq & (\Delta_1+1)g(m-\Delta_1,x,t,\Delta_1),\end{eqnarray*}
where the second and third inequality use the monotonicity properties of $g$. 
\end{proof}

\begin{cor}\label{cor:fg}
There is some positive integer $d \leq n$ such that 
$$f(n,n-t,t) \leq (d+1)g(n,n-t,t,d).$$
\end{cor}

The next lemma gives a lower bound on the maximum missing degree if every $t$-set misses at least $x$ edges. 

\begin{lem}\label{degree}
For any graph on $n \geq t$ vertices such that every $t$-set misses at least $x$ edges, the maximum missing degree $\Delta$ is at least $2nx/t^2$. 
\end{lem}
\begin{proof}
The proof is a simple averaging argument. Since each $t$-set misses at least $x$ edges, by averaging over all $t$-sets, the number of missing edges in the graph is at least $\frac{x}{{t \choose 2}}{n \choose 2} \geq xn^2/t^2$. The maximum missing degree thus satisfies 
 \beq 
\Delta \geq 2\left(xn^2/t^2\right)/n = 2xn/t^2. 
\eeq
\end{proof}

The main idea behind the next lemma is to try iteratively apply Lemma \ref{glemma}, which gives a recursive bound on $g(m,x,t,d)$ until the value of $m$ is at most $t$. 

\begin{lem}
Let $m,x,t,d$ be positive integers and $\Delta=2xm/t^2$. There exists an integer $D$ with $2x/t \leq D \leq d$ such that 
\begin{numcases} 
{g(m, x ,t, d) \leq}  
2^{O(\log^2 \Delta)}
 \left(\prod_{h=D+1}^{\lceil \Delta \rceil}(h+1)^{\frac{t^2}{2xh}}\right)(D+1)^{\frac{t^2}{2x} - \frac{t}{D}}g(t,x,t,D),    &  $\text{if } D \leq \Delta$; \label{nextg} \\
(D+1)^{\frac{m-t}{D}+1}g(t, x, t, D)  ,    & $\text{if } D > \Delta$. \label{nextg2} 
\end{numcases}
\end{lem}
\begin{proof}
We repeatedly apply Lemma \ref{glemma} until the value of $m$ is at most $t$. We start with $n_0 = m$ and $g(m,x,t,d)$ is realized by a graph $G_1$ with $n_1$ vertices and maximum missing degree $\Delta_1$ (so $n_1 + \Delta_1 \leq n_0 = m$). Therefore $g(m,x,t,d) = g(n_1 + \Delta_1, x, t, \Delta_1) \leq (\Delta_1+1) g(n_1, x, t, \Delta_1)$. Let $g(n_1, x, t, \Delta_1)$ be realized by a graph $G_2$ with $n_2$ vertices and maximum missing degree $\Delta_2$ (so $n_2 + \Delta_2 \leq n_1 = m_1$). Therefore $g(n_1,x,t,\Delta_1) = g(n_2 + \Delta_2, x, t, \Delta_2) \leq (\Delta_2+1) g(n_2, x, t, \Delta_2)$. We repeat this process, each time $g(n_j, x, t, \Delta_j)$ is realized by a graph $G_{j+1}$ with $n_{j+1}$ vertices and maximum missing degree $\Delta_{j+1}$. Let $L$ be the first time $n_L \leq t$. We show the following claim.

\begin{claim}
\begin{equation}
g(m,x,t,d)  \leq (\Delta_{L-1}+1)(\Delta_{L-1}+1) g(t,x,t,\Delta_{L-1}). \label{gupp} 
\end{equation}
The set of inequalities are
\begin{align} & d \geq \Delta_1 \geq \Delta_2 \geq \ldots \geq \Delta_L,  \ \ \  n_{L-1} >t,  \ \ \ n_L\leq t;  \nonumber \\
& \text{and for all } j \leq L, n_j + \Delta_j \leq n_{j-1} ,\  \ \ \Delta_j \geq \frac{2n_j x}{t^2}; \label{nd}\\
& \text{In particular,  let $D = \Delta_{L-1}$. Then }  2x/t \leq D \leq d. \nonumber 
\end{align}
\end{claim}
\begin{proof}[Proof of the Claim]
\begin{eqnarray}
g(m,x,t,d) \notag & \leq & (\Delta_1+1)(\Delta_2+1)\cdots (\Delta_{L-2}+1)(\Delta_{L-1}+1)g(n_{L-1},x,t,\Delta_{L-1}) \\ 
\notag & = & (\Delta_1+1)(\Delta_2+1)\cdots(\Delta_{L-2}+1) (\Delta_{L-1}+1)g(n_L + \Delta_L, x, t, \Delta_L) \\
\notag & \leq &(\Delta_1+1)(\Delta_2+1)\cdots (\Delta_{L-2}+1)(\Delta_{L-1}+1) (\Delta_L+1) g(n_L, x, t, \Delta_L) \\
& \leq & (\Delta_1+1)(\Delta_2+1)\cdots(\Delta_{L-2}+1) (\Delta_{L-1}+1)(\Delta_{L-1}+1) g(t,x,t,\Delta_{L-1}),\end{eqnarray}
The last inequality holds because $\Delta_L  \leq \Delta_{L-1}$, $n_L \leq t$, and the monotonicity of $g(m,x,t,d)$.  
 All the inequalities except the last one in (\ref{nd}) is by the definition of the peeling process. The last inequality in (\ref{nd}) is by Lemma \ref{degree}.

If $D=\Delta_{L-1}$, then in particular $D \leq d$, and also $D \geq \frac{2n_{L-1} x}{t^2} \geq \frac{2tx}{t^2} = \frac{2x}{t}$ as $n_{L-1} \geq t$. 
\end{proof}

We are left with obtaining an upper bound for (\ref{gupp}). Fixing $D$, and under the constraints given by the inequalities in (\ref{nd}), we first make the following claim. 

\begin{claim} \label{claim:upbound1}
The maximum of $(\Delta_1+1)(\Delta_2+1)\cdots (\Delta_{L-1}+1)$ is achieved by $n_j + \Delta_j = n_{j-1}$ for all $j \leq L-2$. 
\end{claim}
The proof of this claim is by a standard local adjustment argument, and we leave it to Appendix \ref{sec:addproof} to not distract the reader from the flow of the proof. 

We now optimize the value of $g(m, x, t,d)$ by adjusting the values of $\Delta_i$'s. 

Note that each factor $(\Delta_j+1)$ decreases the $m$ value by $\Delta_j$ until $m = n_0$ drops to $n_{L-2}$, and we can view this as decreasing $m$ by one $\Delta_j$ times, each time giving a factor $(\Delta_j+1)^{1/\Delta_j}$. Since $(y+1)^{1/y}$ is a monotone decreasing function, it follows that (\ref{gupp}) is maximized if, for each $h$, as many of the $\Delta_j$ are no more than $h$. This implies that, given the $\Delta_{j'}$ values which are less than $h$, as many of the $\Delta_j$'s are equal to $h$ as possible to maximize the product. 

Given $n_{L-2}$ and $D$, we want to see how many $\Delta_j$ with $j \leq L-2$ can be equal to $D$. Notice that by Lemma \ref{degree}, $\Delta_j \geq D+1$ whenever $\frac{2n_jx}{t^2} > D$. This means $\Delta_j = D$ only when $n_j \leq \max(\frac{t^2D}{2x}, m)$. Recall from the lemma statement that $\Delta:= \frac{2xm}{t^2}$. We have two cases to consider, depending on whether or not $D>\Delta$. 

In the first case, $D >\Delta$, in which case $\frac{t^2D}{2x}> m$, we only need factors of the form $(D+1)$ and they appear $(m-n_{L-2})/D$ times for $j \leq L-2$. Since $n_{L-2} \geq n_{L-1} + D \geq t+D$,  (\ref{gupp}) is bounded above by 
\begin{align}  (D+1)^{(m-n_{L-2})/D} (D+1) (D+1) g(t, x, t, D)  
\leq & (D+1)^{\frac{m-(t+D)}{D}} (D+1) (D+1) g(t, x, t, D) \nonumber \\
= & (D+1)^{\frac{m-t}{D}+1}g(t, x, t, D) . \label{ineqcase1}
\end{align}

 In the other case, $D \leq \Delta$. We have the following claim.
\begin{claim}\label{claim:DlessDelta}
When  $D \leq \Delta$,
\[(\Delta_1+1)(\Delta_2+1)\cdots (\Delta_{L-2}+1)(\Delta_{L-1}+1) 
\leq 
 2^{O(\log^2 \Delta)} \left(\prod_{h=D+1}^{\lceil \Delta \rceil}(h+1)^{(\frac{t^2}{2x})/h}\right)(D+1)^{\frac{t^2}{2x} - \frac{t}{D}}.
 \]
\end{claim}
Again, the proof is by some standard but technical optimization over the $\Delta_i$'s. We leave the proof to Appendix \ref{sec:addproof}.

Hence, from inequality (\ref{gupp}) we obtain (\ref{nextg}), which completes the proof. 
\end{proof}

Note that the graph realizing $g(t,x,t,D)$ has at most $t$ vertices and each vertex has missing degree at most $D$. The following lemma therefore gives a bound on $g(t,x,t,D)$. 

\begin{lem}\label{boundt}
Let $D$ be a positive integer. If $H=(V,E)$ is a graph on $t$ vertices with at least $x$ missing edges and each vertex has degree at most $D$ in the complement, then the clique number of $H$ is at most $t - \frac{x}{D}$, and the number of cliques in $H$ is at most $2^{t-\frac{x}{D}}\left(1+2^{-1/D} \right)^{\frac{x}{D}}$. 
\end{lem}
\begin{proof}

If $D = 1$, then the result is trivial and sharp since the graph in this case is the complement of a matching with $x$ edges. Now suppose $D>1$. 
We first bound the clique number of $H$. Let $W$ be a maximum clique in $H$, and $\omega$ denote the order of $W$.  Suppose there are $q$ missing edges between $W$ and $V \setminus W$, and $l$ missing edges within $V \setminus W$. Those are all the missing edges since $W$ is a clique. It follows that $q+l \geq x$.
Let $k = t-\omega = |V\setminus W|$. 
Let $v_1, \dots, v_k$ be the vertices in $V \setminus W$ and $S_i$ be the set of vertices in $W$ not adjacent to $v_i$. Thus $1 \leq |S_i| \leq D$ for each $i$, where the upper bound is by the degree restriction and the lower bound is because $v_i$ is not complete to $W$ as $W$ is a maximum clique. As each vertex $v_i \in V \setminus W$ is nonadjacent to at most $D$ other vertices and so at most $D-|S_i|$ other vertices in $V \setminus W$, then by double counting, we have $l \leq  \frac{1}{2} \sum_{i=1}^k(D- |S_i|)$. We also have $q =\sum_{i=1}^k |S_i|$. Since $q + l \geq x$ we have that 
\[ x \leq \frac{1}{2} \sum_{i=1}^k(D- |S_i|) + \sum_{i=1}^k |S_i|,\] which implies 
$x \leq \frac{1}{2}\sum_{i=1}^k \left(D+|S_i|\right)$.
Since $1\leq |S_i| \leq D$, we have $x \leq kD$. Hence, $\omega = t- k \leq t - \frac{x}{D}$. 

The cliques in $H$ can be enumerated as follows. For each clique $\{v_{i_1}, v_{i_2}, \dots, v_{i_j} \}$ in $V \setminus W$, let $N$ denote its common neighborhood in $W$; the number of cliques in $H$ containing the chosen clique in $V \setminus W$ is $2^{|N|}$ given by adding any subset of $N$. Notice that $N$ has size $|W| - |S_{i_1} \cup \dots \cup S_{i_j}| = \omega  - |S_{i_1} \cup \dots \cup S_{i_j}|.$
Therefore, the number of cliques in $H$ is
\beq \sum_{v_{i_1}, \dots, v_{i_j} \text{ forms a clique in } V \setminus W}\!\!\!\!
\!\!\!\! \!\!\!\!
\!\!\!\! \!\!\!\!
\!\!\!\!  2^{\omega - |S_{i_1} \cup S_{i_2} \dots \cup S_{i_j}|}. \label{cliquebd1}
\eeq

Since each vertex in $W$ has at most $D$ non-neighbors in $V \setminus W$, we have 
\beq |S_{i_1} \cup S_{i_2} \dots \cup S_{i_j}| \geq \left(|S_{i_1}| + \dots + |S_{i_j}|\right) / D. \label{boundsize} \eeq
Substituting into (\ref{cliquebd1}), the number of cliques in $H$ is 
\beq 
\sum_{v_{i_1}, \dots, v_{i_j} \text{ forms a clique in } V \setminus W} \!\!\!\!
\!\!\!\! \!\!\!\!
\!\!\!\! \!\!\!\!
\!\!\!\! \!\!2^{\omega - |S_{i_1} \cup S_{i_2} \dots \cup S_{i_j}|} 
 \leq  \! \! \!  \! \! \!\sum_{v_{i_1}, \dots, v_{i_j} \in V \setminus W} \!\!\!\! \!
\!\!\!\! \!\!\!\! 2^{\omega - \left(|S_{i_1}| + |S_{i_2}| +  \dots +| S_{i_j}|\right)/D}  
= 2^\omega \prod_{i=1}^{t - \omega} (2^{-|S_i|/D}+1). \label{averageclique}
\eeq

To maximize (\ref{averageclique}), we fix $\omega$ first, and notice that the number of missing edges between $V \setminus W$ and $W$ is $\sum |S_i| = q.$
We want to maximize $ \prod_{i=1}^{t - \omega} (2^{-|S_i|/D}+1)$ under the constraint that $1 \leq |S_i| \leq D$ and $\sum_{i=1}^{t-\omega} |S_i| = q$. Note that the function $F(y)=\ln(2^{-y}+1)$ is convex. Letting $y_i=|S_i|$, we wish to maximize $\prod_{i}(2^{-y_i}+1)=e^{\sum_{i} \ln(2^{-y_i}+1)}$ 
 constrained to $\sum_i y_i=q$ and $1 \leq y_i \leq D$. From Jensen's inequality, it follows that  $ \prod_{i=1}^{t - \omega} (2^{-|S_i|/D}+1) \leq  (2^{-1/D}+1)^a(2^{-D/D}+1)^b$ where 
$a+b=t-\omega$ and $a+bD=q$. We thus obtain $b=(q-t+\omega)/(D-1)$ and $a= (t-\omega)\frac{D}{D-1}-\frac{q}{D-1}$.

Hence, the number of cliques in $H$ is at most 
\begin{align}
2^{\omega}(2^{-1/D}+1)^a(3/2)^b = 2^{\omega}(1+2^{-1/D})^{ (t-\omega)\frac{D}{D-1}-\frac{q}{D-1}}(3/2)^{(q-t+\omega)/(D-1)}. \label{final}
\end{align}

The right hand side of (\ref{final}) is an increasing function of $\omega$ and a decreasing function of $q$. We already know that $\omega \leq t-\frac{x}{D}$. Furthermore, $q \geq t-\omega$ as each vertex in $V \setminus W$ is not adjacent to at least one vertex in the clique $W$ since $W$ is the largest clique. Therefore $q \geq t  - (t - \frac{x}{D}) = \frac{x}{D}$. So substituting in these values gives us $a=x/D$ and $b= 0$, and so the number of cliques in $H$ is at most 
$2^{t - \frac{x}{D}}  \left(1 + 2^{-1/D} \right)^{x/D}$.
\end{proof}

Combining (\ref{nextg}) and Lemma \ref{boundt}, the following lemma is a direct consequence. 

\begin{lem}
There is an integer $D$ with  $2x/t \leq D \leq d$ such that  
\begin{numcases} 
{g(m, x ,t, d) \leq}  
2^{O(\log^2 \Delta)}
 \left(\prod_{h=D+1}^{\lceil \Delta \rceil}(h+1)^{\frac{t^2}{2xh}}\right)(D+1)^{\frac{t^2}{2x} - \frac{t}{D}}2^{t - \frac{x}{D}}  \left( 1 + 2^{-1/D} \right)^{x/D}    &  $\text{if } D \leq \Delta$,  \ \ \ \ \ \ \  \ \ \ \label{whole1}\\
(D+1)^{\frac{m-t}{D}+1}2^{t - \frac{x}{D}}  \left( 1 + 2^{-1/D} \right)^{x/D}    & $\text{if } D > \Delta$,  \label{whole2}
\end{numcases}
where we recall $\Delta=2xm/t^2$.
\end{lem}

We now have a decent bound on $g(m,x,t,d)$, which we will use to bound $f(n, x, t)$. Recall that $f(n, x, t)$ is what we really want to bound, and the bound we desire is Lemma \ref{boundrec}, which is the main lemma in this subsection. We copy the statement of Lemma \ref{boundrec} for the convenience of the reader. 

\begin{lem*}\nonumber
For $t \leq n \leq 2^{o(t)}$, we have 
\[f(n,n-t,t) \leq 2^{3t+o(t)}.\] 
\end{lem*}

We bound the function $f$ using the function $g$ in the proof.
 
\begin{proof}
Let $n=Ct$ (where $C$ might depend on $t$), $x=n-t=(C-1)t$, and $G$ be a graph on $n$ vertices and maximum missing degree $d$ such that every $t$-set misses at least $x$ edges. Then $2x/t=2(C-1)$ and $\Delta=2xn/t^2=2C(C-1)$. 
Our goal here is to prove
\[
f(Ct, (C-1)t, t) \leq 2^{3t + o(t)}.
\]
Corollary \ref{cor:fg} states that 
$$f(n,n-t,t) \leq (d+1)g(n,n-t,t,d).$$ 
Using this, we want to maximize the logarithm of the bound in (\ref{whole1}) and (\ref{whole2}).
Since $d \leq n = 2^{o(t)}$, thus $\log d  = o(t)$. So we just need to bound $$\log f(n,n-t,t) \leq  \log (d+1) + \log g(n,n-t,t,d) = o(t) + \log g(n, n-t, t,d).$$ It remains to bound $\log g(n, n-t, t,d)$), which we do by a direct application of (\ref{whole1}) and (\ref{whole2}) with a suitable substitution. 

We first bound the logarithm of (\ref{whole2}) over $D \geq \Delta = 2C (C-1)$. Equivalently, we want to bound the logarithm of 
\[ (D+1)^{\frac{(C-1)t}{D}+1}2^{t - \frac{(C-1)t}{D}}  \left( 1 + 2^{-1/D} \right)^{(C-1)t/D}.\]
\begin{claim}\label{claim:case1largeD}
When $D \geq \Delta = 2C (C-1)$, $$f(Ct,(C-1)t,t) \leq 2^{1.64t + o(t)}.$$ 
\end{claim}
It is equivalent to show that the logarithm of (\ref{whole2}) is at most $1.64t + o(t)$. The proof uses standard optimization analysis and we leave the proof to Appendix \ref{sec:addproof}.

We now bound (\ref{whole1}) in the case that $D \leq 2C(C-1)$, as given by the following claim.
\begin{claim}\label{claim:case1smallD}
When $D \leq \Delta = 2C (C-1)$, $$f(Ct,(C-1)t,t) \leq 2^{3t + o(t)}.$$ 
\end{claim}

By (\ref{whole2}), we have (ignoring the $o(t)$ term on the right hand side)
\begin{align}
\log_2 f(Ct, (C-1)t, t) \leq \frac{t^2}{2x} \sum_{y = D+1}^{\lceil \Delta \rceil} \frac{1}{y} \log_2(y+1) + \left(\frac{t^2}{2x}-\frac{t}{D}\right)\log_2(D+1) + t-(1-\log_2 (1 + 2^{-\frac{1}{D}}))\frac{x}{D}.  \label{form1}
\end{align}
It suffices to show that $3t + o(t)$ is an upper bound for the right hand side of (\ref{form1}). A simple way to provide an upper bound is to bound each of the summands above. For the term $\sum_{y = D+1}^{\lceil \Delta \rceil} \frac{1}{y} \log_2(y+1)$ in the first summand, we treat it as a Riemann sum and can thus upper bound  it by an integral.  The rigorous proof of this claim is technical while standard, and is in Appendix \ref{sec:addproof}.

Combining Claims \ref{claim:case1largeD} and \ref{claim:case1smallD}, we complete the proof of Lemma \ref{boundrec}.
\end{proof}

\subsection{Proof of Theorem \ref{main2}}\label{fin2}
Let $G$ be a graph on $n$ vertices with no $K_t$-subdivsion. We apply the peeling process described in Subsection \ref{finimmersion} to the graph $G$. We bound the number of cliques by this peeling process. Let $K$ be a clique of $G$ on $s$ vertices. Let $n_i$ denote the number of vertices in the graph $G_i$ obtained in the peeling process. Let $r=r(K)$ be the least positive integer such that $n_r \leq 1.05t$ or $n_{r+1} \geq n_{r}-n_r^{0.55}$ or $r=s$. 

We first give a bound on $r$. The result of \cite{KS} and \cite{BT} implies, as $G$ does not contain a $K_t$-subdivision, that every subgraph has a vertex of degree at most $d:=t^2$. (here we assume $t$ is sufficiently large).  Hence, $n_1 \leq d+1 = t^2 + 1$. 

Notice that when $n_i$ is large, as $G_i$ does not contain any $K_t$-subdivision, the bound of $\Delta$ in Lemma \ref{subdp} implies that the maximum missing degree in $G_i$ has to be large, and thus by peeling off the vertex with the maximum missing degree, $n_{i+1}$ is much smaller than $n_i$. More specifically, we have the following lemma. 
\begin{lem} \label{dontstop}
If $n_i \geq t^{5/3}$ and $i < s$, then $n_{i} - n_{i+1} \geq n_i^{0.55}$. That is, the procedure does not stop before $n_i$ drops down to $ t^{5/3}$ unless $i =s$.
\end{lem}
\begin{proof}
If it stops at step $i$ and the graph $G_i$ has at most $t^{5/3}$ vertices, it means that the maximum missing degree in $G_i$ satisfies 
\beq \Delta \leq n_i^{0.55} \leq (t^{5/3})^{0.55}.\label{upperd} \eeq
On the other hand, by Lemma \ref{subdp}, we have $\Delta \geq    \frac{2(n_i-t)(n_i-1)}{4(n_i-1)+t(t-1)}.$ 
Plugging in $n_i \geq t^{5/3}$, we have 
\beq  \Delta \geq   \frac{2(n_i-t)(n_i-1)}{4(n_i-1)+t(t-1)} \geq  \frac{2(t^{5/3}-t)(t^{5/3} -1)}{4(t^{5/3}-1) +t(t-1)} \geq t^{4/3} \text{ when } t \gg 1. \label{lowd} \eeq
However, (\ref{upperd}) and (\ref{lowd}) contradict each other, which completes the proof. 
\end{proof}

We bound the number of steps it takes for the procedure to stop. 
Before dropping down the number of remaining vertices to $ n_1/2$, in each step the number of remaining vertices drops by at least  $(n_1/2)^{0.55}$, so it takes at most $\frac{n_1/2}{(n_1/2)^{0.55}} = (n_1/2)^{0.45}$ steps after picking the first vertex to drop the number of vertices to at most $n_1/2$. Similarly, it takes at most $(n_1/2^2)^{0.45}$ further steps to drop the number of vertices from $n_1/2$ to at most $n_1/4$. Continuing, the number of steps it takes before the procedure stops is at most
\beq  
1+(n_1/2)^{0.45} + (n_1/2^2)^{0.45} +  ( n_1/2^3)^{0.45} + \dots = 1+ ( n_1)^{0.45} \frac{1}{2^{0.45}-1} \leq 3 t^{0.9} := r_0. 
\eeq 

We next give a crude but simple bound on the number of choices for $v_1,\ldots,v_r$. We have we have $n_0=n$ choices for $v_1$. We will use the weak estimate that the number of choices for $v_2,\ldots,v_r$ after having picked $v_1$ is at most 
$${|G_2|-1 \choose \leq r_0} \leq r_0 \binom{t^2}{r_0} \leq r_0 \left(\frac{e t^2}{r_0}\right)^{r_0} \leq  2^{7t^{0.9}\log t} \leq  2^{t^{0.91}}.$$
 Therefore there are at most $n2^{t^{0.91}}$ choices for $v_1,\ldots,v_r$. 

Recall that $G$ has $n$ vertices and no $K_t$-subdivision and our goal is to bound the number of cliques in $G$. We have already bounded the number of choices for the first $r$ vertices, and it suffices to bound the number of choices for the remaining vertices. We split the cliques into three types: those with $n_r \leq 1.05t$, those with $r=s$ and $n_r > 1.05t$, and those with $r<s$ and $n_r>1.05t$ and $n_r - n_{r+1} < n_r^{0.55}$ (and thus $t^{5/3}\geq  n_r$ by Lemma \ref{dontstop}).

We first bound the number of cliques with $n_r \leq 1.05t$. As there are at most $1.05t$ possible remaining vertices to include after picking $v_1,\ldots,v_r$ for the clique, there are at most $2^{1.05t}$ ways to extend these vertices. Thus there are at most $n2^{1.05t+t^{0.91}}$ cliques of the first type. We next bound the number of cliques with $r=s$. We saw that this is at most $n2^{t^{0.91}}$. Finally, we bound the number of cliques with  $t^{5/3}\geq n_r \geq 1.05t$, $r<s$, and $n_{r+1} \geq n_r-n_r^{.55}$. In this case, in $G_r$, $v_r$ has the minimum degree, and $v_r$ and its non-neighbors are not in $G_{r+1}$, which has $n_{r+1}$ vertices. Thus, the complement of $G_r$ has maximum degree $\Delta \leq n_{r}-n_{r+1} \leq n_r^{0.55} \leq (t^{5/3})^{0.55}= t^{11/12} = o(t)$. 

By Lemma \ref{subdp} and $\Delta = o(t)$, we have $\sigma(G) \leq t(G)\leq \sigma(G)  + o(t).$ Therefore, $t(G)$ and $\sigma(G)$ differ by at most $o(t)$.
Since $n_r \leq t^{5/3} = 2^{o(t)}$, we can apply Lemma \ref{boundrec}.  Together with the fact that any $t(G)+1 $ vertices in $G_{r}$ contain at least $n_r - t(G)$ edges, we have $G_r$ has at most $2^{3  t(G)+o(t)} \leq 2^{3(\sigma(G) + o(t)) + o(t)} \leq 2^{3t + o(t)}$ cliques as $\sigma(G) \leq t$. Thus the number of cliques in $G_r$ is at most $2^{3t+o(t)}$. This gives a bound on the number of ways of completing a clique of the last type in $G$ having picked the first $r$ vertices. We thus get at most $2^{3t + o(t) + t^{0.91}}n$ cliques of the last type. Adding up all possible cliques, we get at most $2^{3t+o(t)}n$ cliques in $G$, completing the proof. \qed

From the argument we can see that the main contribution for the dependence on $t$ in the upper bound on the number of cliques in a $K_t$-subdivision free graph on $n$ vertices is from the upper bound on $f(n,n-t,t)$ in Lemma \ref{boundrec}. To improve the exponential constant for the number of cliques, we need to bring down the exponential constant in the upper bound for $f(Ct, (C-1)t, t)$. We used nonoptimal approximations that can be improved. For example, this is done in the proof of Claim \ref{claim:case1smallD} as given in Appendix \ref{sec:addproof}, which is part of the proof of Lemma \ref{boundrec}, which as stated gives the constant factor $3$ in the exponent. However, as is clear from the proof,  such as when finding an upper bound for (\ref{form1}), during the computation in (\ref{plugd}), every step we use an estimation; and one can improve them by using higher degree Taylor expansion approximations. Also, step (\ref{plugd}) naively choose the value for $D$ individually for different summands. The presentation is made this way for clarity, but the bound can be improved by more carefully bounding all the summands together. 

Indeed, with the aid of a computer, we can directly find the maximum value for (\ref{form1}) instead of going through multiple steps of crude estimation in (\ref{plugd}). In this way, we can improve the constant in the exponent down to $1.817$. The Taylor expansion computation is involved and tedious and thus we omit it here. Instead, we provide our Python code and output in Appendix \ref{subsec:816} for this optimization. In summary, by either Taylor expansion computation or computer-aided proof, the number of cliques in a graph with $n$ vertices and no $K_t$-subdivision is at most $n 2^{1.817t + o(t)}$. 

\section{Concluding remarks} 

We determined the number of cliques in a graph on $n$ vertices with no $K_t$-immersion up to a factor $O(t^{\log t})$. It would be interesting to determine the exact value, which we conjecture for $n \geq t-2$ to be $2^{t-2}(n-t+3)$ as given by the construction described in the introduction. One approach toward improving the upper bound is to observe that the upper bound proof using the peeling process shows that in an extremal graph, there is a clique of order $t-O(\log^2 t)$. It would be interesting to try to understand the adjacencies from the remaining vertices using the fact that the graph has no $K_t$-immersion. 

We have improved the upper bound on the number of cliques in a graph on $n$ vertices with no $K_t$-subdivision. We were not able to determine the best exponential constant. If $n$ is even and just less than $4t/3$, then the complement of a perfect matching on $n$ vertices has $3^{n/2}$ cliques (getting a factor $3$ for each of the $n/2$ missing edges) and has no $K_t$-subdivision. Taking disjoint unions of such graphs, we get for $n \geq 4t/3$ that there is a graph on $n$ vertices with no $K_t$-subdivision which has $\Omega(n3^{2t/3}/t)$ cliques. We think this construction is essentially optimal  as we proved for minors in \cite{minor}, which is stated as the conjecture below. 

\begin{conj}\label{conj1}
The number of cliques in a graph $G$ with $n$ vertices and no $K_t$-subdivision is at most $n 3^{2t/3+o(t)}$, which would be optimal for $n \geq 4t/3$. 
\end{conj}

There are several ways we thought of that one might be able to use to improve the exponential constant for forbidden clique subdivisions. First, one can try to improve the exponential constant based on the following idea. Originally, instead of Lemma \ref{glemma}, we wanted to use the simple bound $f(n,x,t,d) \leq f(n-1,x,t,\Delta)+f(n-\Delta-1,x,t-1,\Delta)$ for some $\Delta \leq d$, but found it more difficult to rigorously analyze and introduced the function $g$ instead of $f$ for this purpose. This simple bound is proved by considering the cliques that do not contain a vertex $v$ of maximum missing degree $\Delta$ (giving the first term), and those that do contain $v$ (giving the second term). Forgetting about the $x$, $t$, and $d$ values, we get an upper bound which looks like a linear recurrence relation in $n$, suggesting that $f$ should grow in this range at most like an exponential function with exponential constant being a root of the polynomial $p(y)=y^{\Delta+1}-y^{\Delta}-1$. For $\Delta$ large, this root is approximately $1+\frac{\ln \Delta}{\Delta}$. If this part can be made rigorous, then we would get a further improvement which gives the exponential constant 1.458. 
Second, in Lemma \ref{boundrec}, we used Lemma \ref{boundt} to bound the number of cliques in the remaining induced subgraph. We think this lemma should be improvable to get a better exponential constant. Third, in Lemma \ref{glemma}, we obtained a recursive bound on $g(m,x,t,d)$. For the terms in the summation apart from when $i=0$, the proof shows that the vertices we choose are in the neighborhood of a vertex $v_i$, and the value of $t$ actually decreases by one. We didn't use this decrease in $t$, which should accumulate over many applications of this lemma to improve the exponential constant. Finally, for a graph to contain many cliques and not have a $K_t$-subdivision (and in particular, every $t$-set should have many missing edges in our argument) should force structural information about the graph that could be used to improve the exponential constant further. It is unclear if such attempts worked whether it would get to the conjectured exponential constant.  

\begin{appendices}

\section{Additional Proofs}\label{sec:addproof}
\begin{customclaim}{\ref{claim:upbound1}}
The maximum of $(\Delta_1+1)(\Delta_2+1)\cdots (\Delta_{L-1}+1)$ is achieved when $n_j + \Delta_j = n_{j-1}$ for all $j \leq L-2$. 
\end{customclaim}
\begin{proof}
The proof is by local adjustment. 

We first prove it holds for all $j$ with $\Delta_j > \Delta_{L-2}$. Suppose we have already achieved the optimal. Let $K$ be the smallest integer such that $n_K + \Delta_K < n_{K-1}$. If $\Delta_K = \Delta_{L-2}$, then by the definition of $K$ and the monotonicity of the $\Delta_i$'s, we know
$K > j$ for all $j$ with $\Delta_j> \Delta_{L-2}$, and we are done. So we may assume $\Delta_K > \Delta_{L-2}$. 
Let $l$ be the largest integer such that $\Delta_l = \Delta_K$. Since $\Delta_K > \Delta_{L-2}$, we have that $\Delta_{l+1} < \Delta_K$. If $K = 1$ or $\Delta_K < \Delta_{K-1}$, then we can increase $\Delta_K$ by 1 and the remaining $\Delta_j$'s and $n_j$'s unchanged; the inequalities in (\ref{nd}) still hold while the product increases. This is a contradiction. Therefore $K>1$ and  $\Delta_K = \Delta_{K-1}$.  
Since $\Delta_K > \Delta_{L-2}$, we have $l+1 \leq L-2$. 
We claim that we can let $\Delta_{l+1}' = \Delta_{l+1}+1$ and $n_j' = n_j+1$ for all $j$ with $K \leq j \leq l$. To see this we need to check the inequalities in (\ref{nd}). The inequalities indeed hold: for $K+1 \leq j \leq l$, we have $\Delta_{j} + n_{j}' \leq n_{j-1}'$; $\Delta_K + n_K' \leq n_{K-1}$ since $\Delta_K + n_K < n_{K-1}$ before the change;  $\Delta_{l+1}' \geq  \frac{2n_{l+1}}{t^2}$; and, for $j \leq K \leq l$, we have $\Delta_j \geq \Delta_K = \Delta_{K-1} \geq \frac{2n_{K-1}}{t^2} \geq \frac{2(n_K+1)}{t^2} \geq \frac{2n_j'}{t^2}$. While the product increases since some $\Delta_j$ increase. This leads to a contradiction and thus we have proved $n_j + \Delta_j = n_{j-1}$ for all $j$ with $\Delta_j > \Delta_{L-2}$. 

Now, we proceed to prove $n_j + \Delta_j = n_{j-1}$ for all $j$ with $\Delta_j = \Delta_{L-2}$. If $\Delta_j = \Delta_{L-2}$ for $K \leq j \leq L-2$ while $n_K+ \Delta_K < n_{K-1}$, by the same argument as above we have $\Delta_K = \Delta_{K-1}$. Thus we let $n_j' = n_K+1$ for all $K \leq j \leq L-2$. To check the inequalities in (\ref{nd}), similarly for $K \leq j \leq L-2$, we have
$\Delta_j = \Delta_K = \Delta_{K-1} \geq \frac{2n_{K-1}}{t^2} \geq \frac{2(n_K+1)}{t^2} \geq \frac{2n_j'}{t^2}$. 
Thus it is clear that all the inequalities in (\ref{nd}) holds.

Therefore we have shown that the optimal is achieved when  $n_j + \Delta_j = n_{j-1}$ for all $j \leq L-2$. 
\end{proof}

\begin{customclaim}{\ref{claim:DlessDelta}}
When  $D \leq \Delta$,
\[(\Delta_1+1)(\Delta_2+1)\cdots (\Delta_{L-2}+1)(\Delta_{L-1}+1) 
\leq 
 2^{O(\log^2 \Delta)} \left(\prod_{h=D+1}^{\lceil \Delta \rceil}(h+1)^{(\frac{t^2}{2x})/h}\right)(D+1)^{\frac{t^2}{2x} - \frac{t}{D}}.
 \]
\end{customclaim}
\begin{proof}
Assume $D \leq \Delta$. The number of $j \leq L-2$ which are equal to $D$ is at most $(\frac{t^2D}{2x} - n_{L-2})/D$. Since $n_{L-2} \geq n_{L-1} + D \geq t+D$, we have that the number of $j \leq L-1$ which are equal to $D$ is at most $(\frac{t^2D}{2x} - (t+D))/D+1 = \frac{t^2}{2x} - \frac{t}{D}$. 

Now suppose we have picked the $\Delta_j$'s which are as small as possible while having values no larger than $h-1$ (for $h \geq D+1$). By the induction hypothesis, we know that once $\frac{2n_jx}{t^2} > h$, we cannot choose $\Delta_j = h$. This means $\Delta_j = h$ corresponds to $n_j \leq \lfloor \frac{t^2 h}{2x} \rfloor$. On the other hand, by induction hypothesis, we know that when $n_j \leq \lfloor \frac{t^2( h-1)}{2x} \rfloor$ we have picked $\Delta_j = h-1$. Therefore the number of $j$'s such that $\Delta_j = h$ is at most $\left(\min(m, \lfloor\frac{t^2 h}{2x} \rfloor) - \lfloor\frac{t^2 (h-1)}{2x} \rfloor\right) / h \leq \min\left(m -\lfloor \frac{t^2 (h-1)}{2x} \rfloor, \frac{t^2 }{2x}+1\right)/h$. 
Notice that in this way, before $m$ drops to one, the largest $h$ that $\Delta_j$ can achieve is when $ \frac{t^2 (h-1)}{2x} < m \leq  \frac{t^2 h}{2x}$, i.e., when $\frac{2mx}{t^2}  \leq h < \frac{2mx}{t^2} + 1 = \Delta + 1$. This means $h = \lceil \Delta \rceil$. There are $\left(m - \lfloor \frac{t^2 (h-1)}{2x} \rfloor\right)/h$ number of $j$'s with $\Delta_j = h= \lceil \Delta \rceil$. 


We therefore obtain the bound   
\begin{align}
&(\Delta_1+1)(\Delta_2+1)\cdots (\Delta_{L-2}+1)(\Delta_{L-1}+1) \nonumber \\
 \leq & \left(\lceil \Delta \rceil+1\right)^{ \frac{1}{\lceil \Delta \rceil}\left(m - \left\lfloor \frac{t^2 ( \lceil \Delta \rceil-1)}{2x} \right \rfloor\right)}\left(\prod_{h=D+1}^{\lceil \Delta \rceil-1}(h+1)^{\frac{1}{h}\left(\frac{t^2}{2x}+1\right)}\right)(D+1)^{\frac{t^2}{2x} - \frac{t}{D}}  \nonumber \\ 
\leq & 
2^{O(\log^2 \Delta)} \left(\lceil \Delta \rceil+1\right)^{ \frac{1}{\lceil \Delta \rceil}\left(m - \left( \frac{t^2 ( \lceil \Delta \rceil-1)}{2x} \right)\right)}\left(\prod_{h=D+1}^{\lceil \Delta \rceil-1}(h+1)^{(\frac{t^2}{2x})/h}\right)(D+1)^{\frac{t^2}{2x} - \frac{t}{D}}  \label{ineqcase2accurate} \\
\leq & 2^{O(\log^2 \Delta)} \left(\prod_{h=D+1}^{\lceil \Delta \rceil}(h+1)^{(\frac{t^2}{2x})/h}\right)(D+1)^{\frac{t^2}{2x} - \frac{t}{D}}  \label{ineqcase2} 
\end{align}
The last inequality holds because $\left(m - \lfloor \frac{t^2 (h-1)}{2x} \rfloor\right)/h < \left(\frac{t^2}{2x}+1\right)/h$.
\end{proof}

\begin{customclaim}{\ref{claim:case1largeD}}
When $D \geq \Delta = 2C (C-1)$, $$f(Ct,(C-1)t,t) \leq 2^{1.64t + o(1)}.$$ 
\end{customclaim}
\begin{proof}
It is equivalent to show that the logarithm of (\ref{whole2}) is maximized by $1.64t + o(t)$. In other words, we want to show that 
the logarithm of 
\[ (D+1)^{\frac{(C-1)t}{D}+1}2^{t - \frac{(C-1)t}{D}}  \left( 1 + 2^{-1/D} \right)^{(C-1)t/D}\]
is bounded above by $1.64t + o(t)$ when $D \geq \Delta = 2C (C-1)$.

By plugging in the values for $n$ and $x$, we have $t^{-1}\log g(n,n-t,t,d)$ in this case (ignoring an additive $o(1)$) is at most
\begin{align}
 & \frac{(C-1)}{D} \log (D+1) + 1 - \frac{(C-1)}{D} + \log(1+2^{-1/D}) \frac{(C-1)}{D} \\
 =   &
1+  (C-1)D^{-1} \left(\log (D+1)   - (1 - \log(1+ 2^{-1/D}))\right). \label{temp1}
\end{align}
Notice that the function of $D$, defined as $h(D) = D^{-1}\left(\log (D+1)   - (1 -  \log(1+ 2^{-1/D}))\right)$, is monotone decreasing for $D \geq 1$. Therefore to maximize (\ref{temp1}) we should choose $D = \Delta = 2C(C-1)$. By plugging in $D = 2C(C-1)$ into (\ref{temp1}), 
we want to maximize
\begin{align}
&1 + (C-1)(2C(C-1) )^{-1} \left(\log(2C(C-1) +1)-\left(1 - \log(1+2^{-1/(2C(C-1) )})\right)\right)  \nonumber \\
& = 1 + (2C)^{-1}\left(\log(2C(C-1) +1)-1 + \log(1+2^{-1/(2C(C-1) )})\right). \label{temp11}
\end{align}

By computing its first derivative as a function of $C$, we obtain that (\ref{temp11}) is bounded above by 1.64 when $D \geq \Delta$. Alternatively, the Mathematica code in Code \ref{code1} in Appendix \ref{subsec:temp11} also shows the result. 
Therefore when $D \geq \Delta$ we have that 
$$f(Ct,(C-1)t,t) \leq 2^{1.64t + o(1)}.$$ 
\end{proof}

\begin{customclaim}{\ref{claim:case1smallD}}
When $D \leq \Delta = 2C (C-1)$, $$f(Ct,(C-1)t,t) \leq 2^{3t + o(1)}.$$ 
\end{customclaim}
\begin{proof}
Recall $n = Ct$ and $x = (C-1)t$.
By (\ref{whole2}), we have (ignoring the $o(t)$ term on the right hand side)
\begin{align}
\log_2 f(Ct, (C-1)t, t) \leq \frac{t^2}{2x} \sum_{y = D+1}^{\lceil \Delta \rceil} \frac{1}{y} \log_2(y+1) + \left(\frac{t^2}{2x}-\frac{t}{D}\right)\log_2(D+1) + t-(1-\log_2 (1 + 2^{-\frac{1}{D}}))\frac{x}{D}.  \label{form}
\end{align}
It suffices to show that $3t + o(t)$ is an upper bound the right hand side of (\ref{form}).

We further ignore $o(t)$ terms, which are negligible, throughout the computation. 
When $y \geq 5$, we have $\log(y-1.95) / (y-1.95) \geq \log(1+y) / y$. We replace $\log(1+y) / y$ by $\log(y-1.95) / (y-1.95)$ in (\ref{form}) when $y \geq 5$. Therefore, when $C \geq 3$, by (\ref{nd}) it is clear that $D+1 \geq 2(C-1)+1 \geq 5$.  Also, $(1 + 2^{-1/D}) \leq 2$. 
Thus, we can obtain from (\ref{form}) by dividing by $t$ 
\begin{align}
& \ \ \ \ \ \ t^{-1}\log_2 f(Ct, (C-1)t, t) \nonumber \\
& \leq  \frac{t}{2x} \sum_{y = D+1}^{\lceil \Delta \rceil } \frac{1}{y-1.95} \log_2(y-1.95) + \left(\frac{t}{2x}-\frac{1}{D}\right)\log_2(D+1) + 1-(1-\log_2 (1 + 2^{-\frac{1}{D}}))\frac{x}{Dt}  \nonumber \\
& \leq \frac{t}{4x} \left(  \log^2(\lceil \Delta \rceil  -2.95)-  \log^2(D -1.95) \right) + \left(\frac{t}{2x}-\frac{1}{D}\right)\log_2(D+1) + 1-(1-\log_2 2)\frac{x}{Dt} \nonumber \\
& \leq 1 + \frac{1}{4(C-1)}(\log^2(2C(C-1)-1.95) - \log^2(2(C-1)-1.95)) + \label{plugd} \\
&  \ \ \ \ \ \left(\frac{1}{2(C-1)} - \frac{1}{2C(C-1)}\right) \log_2(2C(C-1)+1) \nonumber \\
& =  1+ \frac{1}{4(C-1)}(\log^2(2C(C-1)-1.95) - \log^2(2(C-1)-1.95)) +  \frac{1}{2C} \log_2(2C(C-1)+1)  \nonumber  \\
& \leq 2.92,  \nonumber
\end{align}
where the first inequality is by (\ref{form}), the second inequality is by using the inequality $$\sum_{y = i}^j \frac{1}{y+s} \log (y+s) \leq \int_{s+i-1}^{s+j-1} \frac{\log x}{x} dx = \frac{1}{2}(\log^2(s+j-1) - \log^2(s+ i-1))$$ which holds for $i+s \geq 2$ where we substitute in $i=D+1$, $j=\lceil \Delta \rceil$, and $s=-1.95$, the third inequality is by plugging in the appropriate $D$ in the range $2(C-1) \leq D \leq 2C(C-1)$ to maximize each summand individually, and the last inequality can be checked by graphing the two functions or formally by using the Taylor series approximation to the function $\log (1+x)$. For the reader's convenience, we also provide a Mathematica code in Code \ref{code2} in Appendix \ref{subsection:plugd}  to verify the last inequality in (\ref{plugd}). 

If $C \leq 3$, we have $f(Ct, (C-1)t, t) \leq 3t$ as any graph with at most $3t$ vertices has at most $2^{3t}$ cliques. 
\end{proof}

\section{Numerical Computation}\label{app:numerical}
\subsection{Maximize (\ref{temp11})}\label{subsec:temp11}
We use the Mathematica algorithm \emph{Maximize} to find the global maximizer of a given function in the specified domain. 
To maximize (\ref{temp11}), we type in the function (\ref{temp11}) in the first line and specify the range of $C$ to be $C \geq 1$; Mathematica will output the numerical maximizer of this function in the range $C \geq 1$. The second line shows the output from Mathematica. It shows that the the maximum is, after rounded to the fifth digit after the decimal, $1.61098$, with the optimized $C$ rounded to $2.83747$. 
\begin{lstlisting}[language=Mathematica, caption={Mathematica code to maximize (\ref{temp11}) when $D \geq \Delta$}, label={code1}]
In[1] := Maximize[{1 + (2*C )^(-1) *(Log[2*C*(C-1) +1]/Log[2] - (1 - Log[1+2^(-1/(2*C*(C-1) ))]/Log[2])),C >= 1},{C}] \\
Out[1] = {1.61098, {$C->2.83747$}}
\end{lstlisting}

\subsection{Last inequality in (\ref{plugd})}\label{subsection:plugd}
We use the Mathematica algorithm \emph{Maximize} to find the global maximizer of a given function in the specified domain. 
To prove the last inequality in (\ref{plugd}), we type in the function in the second to last line in (\ref{plugd}) and specify the range of $C$ to be $C \geq 3$; Mathematica will output the numerical maximizer of this function in the range $C \geq 3$. The second line shows the output from Mathematica. It shows that the maximum is, after rounded to the fifth digit after the decimal, $2.91048$, with the optimized $C$ rounded to $3.59459$. 
 \begin{lstlisting}[language=Mathematica,caption={Mathematica code to show the last inequality in (\ref{plugd})}, label={code2}]
In[2] := Maximize[{1 + Log[2 x (x - 1) - 1.95]^2/(4 Log[2]^2 (x - 1)) -  Log[2 (x - 1) - 1.95]^2/(4 (x - 1) Log[2]^2) + (1/(2 x)) (Log[2 x (x - 1) + 1]/Log[2]),  x >= 3}, {x}]  \\
Out[2] = {2.91048,{x->3.59459}} 
\end{lstlisting}

\subsection{Computer output to show $f(Ct, (C-1)t, t) \leq 2^{1.817t+o(t)}$}\label{subsec:816}
As mentioned in the last paragraph in Subsection \ref{fin2}, in order to bring the exponent in the bound of $f(Ct, (C-1)t, t)$ from $3$ to 1.817, instead of analyzing 
(\ref{ineqcase2}), we can directly upper bound (\ref{ineqcase2accurate}) with the aid of a computer.
We copy the logarithm of the expression in (\ref{form}) (or equivalently, the first inequality in (\ref{plugd})) here for convenience:
\begin{align*}
 & t^{-1}\log_2 f(Ct, (C-1)t, t)  \\
 \leq &  \frac{t}{2x} \sum_{y = D+1}^{\lceil \Delta \rceil } \frac{1}{y-1.95} \log_2(y-1.95) + \left(\frac{t}{2x}-\frac{1}{D}\right)\log_2(D+1) + 1-(1-\log_2 (1 + 2^{-\frac{1}{D}}))\frac{x}{Dt}. 
 \end{align*}
 Recall that $x = (C-1)t$, and $D \leq \Delta = 2C(C-1)$ and $D \geq 2(C-1)$ by the assumption of Claim \ref{claim:case1smallD} and the conditions in (\ref{nd}). 
 Therefore we just need to study the following optimizing problem:
 \begin{equation} \text{max }
\frac{1}{2(C-1)} \sum_{y = D+1}^{\lceil \Delta \rceil } \frac{1}{y-1.95} \log_2(y-1.95) + \left(\frac{1}{2(C-1)}-\frac{1}{D}\right)\log_2(D+1) + 1-(1-\log_2 (1 + 2^{-\frac{1}{D}}))\frac{C-1}{D}  \label{codeeq}
 \end{equation}
 with all possible positive integers $D$ and positive values $C \geq 1$ with 
\[ \lceil 2(C-1) \rceil \leq D \leq \lfloor 2C(C-1) \rfloor.\] 
 We can add the floor or ceiling in the previous expression because $D$ is an integer. 

We now explain how we implement the computer-aided proof. 
Lines 6 to 17 is to write down the expression. Lines 6 to 10 evaluates  the sum $\sum_{y = D+1}^{\lceil \Delta \rceil } \frac{1}{y-1.95} \log_2(y-1.95)$, and defines it to be the function \emph{sum1} with arguments $D, \Delta$. 

Lines 12 to 17 evaluates the sum (\ref{codeeq}) with given values $C, D, \Delta$, and define the sum as \emph{bigsum}, as named in Line 12. 
Line 11 introduces a variable \emph{val}, which is the sum (\ref{codeeq}). 

From Lines 19 to 35, for each fixed $C$, we enumerate all the possible values of $D$ such that $D$ is a positive integer and  $ \lceil 2(C-1) \rceil \leq D \leq \lfloor 2C(C-1) \rfloor$, and see which value of $D$ gives the maximum value of (\ref{codeeq}), which is also the value of \emph{bigsum}($c, D, \Delta$) in Line 28.

 \begin{python}[caption={Python code to show the last inequality in (\ref{plugd})}, label={code3}]
import math
from decimal import *
context = Context(prec=1000)
setcontext(context)

def sum1(D, Delta): #this is for the first term in equation (16): sum 1/y * log(y+1) when y goes through D+1 to Floor[Delta]
    temp = Decimal(0)
    for i in range(int(D+1), int(Delta+1)):
        temp = temp + Decimal(1.0/i*math.log(i+1, 2))
    return temp

def bigsum(c, D, Delta): # this evaluates the right hand side of equation (16), assuming we know D
    c = Decimal(c)
    D = Decimal(D)
    Delta = Decimal(Delta)
    val = Decimal(Decimal(1.0/2)/(c-1)*sum1(D, Delta) + (Decimal(1.0/2)/(c-1)-1/D)*Decimal(math.log(D+1,2)) + 1 -Decimal(1.0-math.log(1+2**(-1/D),2))*(c-1)/D)
    return val

def optBigsum(c):  # This evaluates the maximum possible value for the right hand side of equation (16), among all feasible values of D (when fixed C)
    Dlow = math.ceil(2*(c-1))       # This is the lower bound for D. Note that D is an integer, thus D >= ceiling(2*(C-1))
    Delta = math.floor(2*c*(c-1))   # This is the upper bound for D. Note that D is an integer, thus D <= floor(2*C*(C-1))
    if Delta < Dlow:
        return (None, None)
    else:
        value = 0
        tempD = Dlow
        for D in range(int(Dlow), int(Delta+1)):
            bigs = bigsum(c,D,Delta)
            
            if bigs> value:
                value = bigs
                tempD = D
                
        return (round(value,3), int(tempD))    
 \end{python}      
In summary, for a specific value of $C$, the last function in this program, \emph{optBigsum}, will output the maximum value for (\ref{codeeq}) and also which value of $D$ is the maximizer. 
 We asks Python to run this code to many values of $C$. 

Below we provide some sample output. We actually ran the program for many more values of $C$ giving a much finer net for the space of possible values. The middle column is the value of  $t^{-1}\log f(Ct, (C-1)t, t)$, which we want to show is bounded above by 1.817. 
 \begin{lstlisting}[title={Output 1}, label={code3op1}]
value C is 1.5;    the value is: 0.792;     the optimized D is: 1
value C is 2.0;    the value is: 1.571;     the optimized D is: 3
value C is 2.5;    the value is: 1.62;     the optimized D is: 4
value C is 3.0;    the value is: 1.751;     the optimized D is: 5
value C is 3.5;    the value is: 1.764;     the optimized D is: 7
value C is 4.0;    the value is: 1.806;     the optimized D is: 8
value C is 4.5;    the value is: 1.803;     the optimized D is: 9
value C is 5.0;    the value is: 1.816;     the optimized D is: 10
value C is 5.5;    the value is: 1.807;     the optimized D is: 11
value C is 6.0;    the value is: 1.807;     the optimized D is: 12
value C is 6.5;    the value is: 1.796;     the optimized D is: 13
value C is 7.0;    the value is: 1.792;     the optimized D is: 14
value C is 7.5;    the value is: 1.78;     the optimized D is: 15
value C is 8.0;    the value is: 1.773;     the optimized D is: 17
value C is 8.5;    the value is: 1.761;     the optimized D is: 18
value C is 9.0;    the value is: 1.753;     the optimized D is: 19
value C is 9.5;    the value is: 1.742;     the optimized D is: 20
value C is 10.0;    the value is: 1.733;     the optimized D is: 21
value C is 10.5;    the value is: 1.722;     the optimized D is: 22
\end{lstlisting}
The output above shows that the value of $t^{-1}\log f(Ct, (C-1)t, t)$ (ignoring an additive $o(1)$) has its maximum value between $C = 3.5$ to $C = 6$. Among the above output, $C = 5.0$ gives the largest value $1.816$.

We have computed even more sample points. Figure \ref{figure:code} is a plot of a subset of the sample points and a superset of the data in Output 1, with more sample points closer to the extremum. The $x$-axis gives the values of $C$, and the $y$-values are $t^{-1}\log f(Ct, (C-1)t, t)$ (ignoring an additive $o(1)$). We can see that the maximum value is when is still 1.816. 
\begin{figure}
\begin{center}
\includegraphics[scale=0.7, trim={2cm 9cm 2cm 9.5cm},clip]{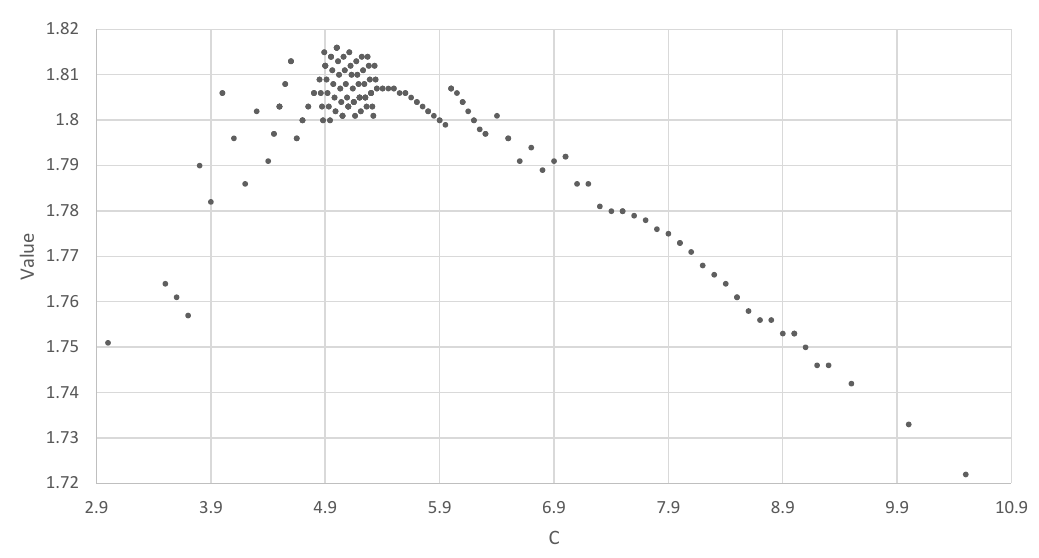}
\caption{Values of $t^{-1}\log f(Ct, (C-1)t, t)$. \footnotesize{the $x$-axis are the values of $C$, and the $y$-values are  $t^{-1}\log f(Ct, (C-1)t, t)$. The highest dot (data point) is given by $C=5.0$, which gives the value 1.816.}}
\label{figure:code}
\end{center}
\end{figure}

For the reader's convenience, we also provide the python source file on the second author's homepage. 

\end{appendices}

\end{document}